\documentclass[11pt]{amsart}
\usepackage{a4wide,amsfonts,amsmath,amssymb,amsthm,epsfig,cite,graphicx,hyperref,
color, esint,fancyhdr, enumerate, latexsym, amsrefs,cleveref}

\newtheorem{thm}{Theorem}[section]
\newtheorem{prop}[thm]{Proposition}
\newtheorem{lem}[thm]{Lemma}

\theoremstyle{definition}

\newcommand{\ep}{\epsilon}
\newcommand{\mbb}{\mathbb}
\newcommand{\de}{\delta}
\newcommand{\ov}{\overline}
\newcommand{\La}{\Lambda}
\newcommand{\pa}{\partial}
\newcommand{\mf}{\mathbb}

\newcommand{\Om}{\Omega}
\newcommand{\al}{\alpha}
\newcommand{\be}{\beta}

\newcommand{\z}{\zeta}
\newcommand{\la}{\lambda}
\newcommand{\ti}{\tilde}
\newcommand{\De}{\Delta}

\renewcommand{\Re}{\operatorname{Re}}

\hypersetup{
    colorlinks,
    citecolor=blue,
    filecolor=black,
    linkcolor=red,
    urlcolor=black
}

\begin{document}
\title{Boundary behaviour of the Carath\'eodory and Kobayashi-Eisenman volume elements}
\thanks{The first named author was partially supported by the DST-INSPIRE grant IFA-13 MA-21.}
\author{Diganta Borah and Debaprasanna Kar}

\address{Diganta Borah: Indian Institute of Science Education and Research, Pune  411008, India}
\email{dborah@iiserpune.ac.in}

\address{Debaprasanna Kar: Indian Institute of Science Education and Research, Pune  411008, India}
\email{debaprasanna.kar@students.iiserpune.ac.in}


\begin{abstract}
We study the boundary asymptotics of the Carath\'eodory and Kobayashi-Eisenman volume elements on smoothly bounded convex finite type domains and Levi corank one domains.
\end{abstract}

\maketitle


\section{Introduction}
Denote by $\mbb{B}^n$ the unit ball in $\mbb{C}^n$. For a domain $D\subset\mbb{C}^n$, the Carath\'eodory and Kobayashi-Eisenman volume elements on $D$ at a point $p \in D$ are defined respectively by
\begin{align*}
c_D(p) & = \sup \Big\{ \big\vert \det \psi^{\prime}(p) \big\vert^2 : \psi \in \mathcal{O}(D,\mbb{B}^n), \psi(p) = 0 \Big\},\\
k_D(p) &= \inf \Big\{\big\vert \det \psi^{\prime}(0)\big\vert^{-2 } : \psi \in \mathcal{O}(\mbb{B}^n, D), \psi(0)=p \Big\}.
\end{align*}
By Montel's theorem $c_D(p)$ is always attained and if $D$ is taut then $k_D(p)$ is also attained. Under a holomorphic map $F:D \to \Om$, they satisfy the rule
\[
v_D(p) \geq \big\vert \det F^{\prime}(p)\big\vert^2 v_{\Om}\big(F(p)\big)
\]
where $v=c,k$. In particular, equality holds if $F$ is a biholomorphism. Accordingly, if $k_D$ is nonvanishing (which is the case if $D$ is bounded or taut), then
\[
q_D(p)=\frac{c_D(p)}{k_D(p)}
\]
is a biholomorphic invariant and is called the \textit{quotient invariant}. If $D=\mbb{B}^n$, then
\[
c_{\mbb{B}^n}(p)=k_{\mbb{B}^n}(p)=\big(1-\vert p \vert^2\big)^{-n-1},
\]
and thus $q_{\mbb{B}^n}$ is identically equal to $1$. In general, an application of the Schwarz lemma shows that $q_D \leq 1$. It is a remarkable fact that if $D$ is any domain in $\mbb{C}^n$ and $q_D(p)=1$ for some point $p \in D$, then $q_D(z)=1$ for all $z \in D$ and $D$ is biholomorphic to $\mbb{B}^n$. This was first proved by Wong \cite{Wong} with the hypothesis that $D$ is bounded and complete hyperbolic, which was relaxed by Rosay \cite{Ro} to $D$ being any bounded domain. Dektyarev \cite{Dek} further relaxed this condition to $D$ being only hyperbolic and later Graham and Wu \cite{Gr-Wu} showed that no assumption on $D$ is required for the result to be true, in fact, it is true for any complex manifold. Thus $q_D$ measures the extent to which the Riemann mapping theorem fails for $D$. This fact is a fundamental step in the proof of the Wong-Rosay theorem and several other applications can be found in \cites{Gr-Kr1, Gr-Kr2, Kr-vol}.

The purpose of this note is to study the boundary asymptotics of the Kobayashi volume element on smoothly bounded convex finite type domains and Levi corank one domains in $\mbb{C}^n$. Boundary behaviour of the quotient invariant on strongly pseudoconvex domains had been studied by several authors, see for example \cites{Cheung-Wong, Gr-Kr1, Ma}, and in particular it is known that $q_D(z) \to 1$ if $z \to \pa D$ for a strongly pseudoconvex domain $D$. Recently in \cite{Nik-sq}, \textit{nontangential} boundary asymptotics of the volume elements near $h$-extendible boundary points were obtained. Finally, we also note that in \cite{Nik-Pas}, a relation between the Carath\'eodory volume element and the Bergman kernel was observed in light of the multidimensional Suita conjecture. Our goal is to compute the boundary asymptotics of the Kobayashi volume element in terms of the distinguished polydiscs of McNeal and Catlin devised to capture the geometry of a domain near a convex finite type and Levi corank one boundary point respectively. In order to state our results, let us briefly recall these terminologies. First, let $D=\{\rho<0\}$ be a smoothly bounded convex finite type domain and $p^0 \in \pa D$. For each point $p \in D$ sufficiently close to $p^0$, and $\ep>0$ sufficiently small, McNeal's orthogonal coordinate system $z_1^{p,\ep}, \ldots, z_n^{p,\ep}$ centred at $p$ is constructed as follows (see \cite{McN-adv}). Denote by $D_{p, \ep}$ the domain
\[
D_{p,\ep}=\Big\{ z \in \mbb{C}^n :\rho(z)< \rho(p)+\ep\Big\}.
\]
Let $\tau_n(p,\ep)$ be the distance of $p$ to $\pa D_{p,\ep}$ and $\zeta_n(p,\ep)$ be a point on $\pa D_{p, \ep}$ realising this distance. Let $H_n$ be the complex hyperplane through $p$ and orthogonal to the vector $\zeta_n(p,\ep)-p$. Compute the distance from $p$ to $\pa D_{p,\ep}$ along each complex line in $H_n$. Let $\tau_{n-1}(p,\ep)$ be the largest such distance and let $\zeta_{n-1}(p,\ep)$ be a point on $\pa D_{p,\ep}$ such that $\vert \zeta_{n-1}(p,\ep)-p\vert=\tau_{n-1}(p,\ep)$. For the next step, define $H_{n-1}$ as the complex hyperplane through $p$ and orthogonal to the span of the vectors $\zeta_n(p,\ep)-p, \zeta_{n-1}(p,\ep)-p$ and repeat the above construction. Continuing in this way, we define the numbers $\tau_n(p, \ep), \tau_{n-1}(p, \ep), \ldots, \tau_1(p, \ep)$, and the points $\zeta_{n}(p, \ep)$, $\zeta_{n-1}(p, \ep)$, $\ldots, \zeta_1(p, \ep)$ on $\pa D_{p, \ep}$. Let $T^{p, \ep}$ be the translation sending $p$ to the origin and $U^{p,\ep}$ be a unitary mapping aligning $\zeta_k(p,\ep)-p$ along the $z_k$-axis and $\zeta_k(p, \ep)$ to a point on the positive $\Re z_k$ axis. Set
\[
z^{p,\ep}= U^{p, \ep} \circ T^{p, \ep}(z).
\]
The polydisc
\[
P(p,\ep)=\Big\{z^{p,\ep} : \vert z^{p,\ep}_1 \vert < \tau_1(p, \ep), \ldots, \vert z^{p,\ep}_n\vert < \tau_{n}(p, \ep)\Big\}
\]
is known as McNeal's polydisc. Write $z = (z_1, z_2, \ldots, z_n) = ({}'z, z_n) \in \mbb C^n$. The scaling method (which will be briefly explained later) shows that every sequence in $D$ that converges to $p^0 \in \pa D$ furnishes limiting domains
\begin{equation}\label{cvx-mod}
D_{\infty} = \left\{ z \in \mbb C^n : -1 + \Re \sum_{\al=1}^n b_{\al} z_{\al} + P_{2m}({}'z)  < 0 \right\},
\end{equation}
where $b_\al$ are complex numbers and $P_{2m}$ is a real convex polynomial of degree at most $2m$ ($m \ge 1$), where $2m$ is the $1$-type of $\pa D$ at $p^0$. The polynomial $P_{2m}$ is not unique in general and depends on how the given sequence approaches $p^0$. The limiting domains $D_{\infty}$ are usually called local models associated with $D$ at $p^0$. It is known that $D_{\infty}$ possesses a local holomorphic peak function at every boundary point including the point at infinity and hence is complete hyperbolic (see \cite{G2}).

\begin{thm} \label{cvx}
Let $D=\{\rho<0\}$ be a smoothly bounded convex finite type domain in $\mathbb{C}^n$ and $p^j \in D$ be a sequence converging to $p^0 \in \pa D$. Let $\ep_j=-\rho(p^j)$. Then up to a subsequence,
\[
k_D(p^j) \prod_{\al=1}^n \tau_{\al}(p^j,\ep_j)^2 \to k_{D_\infty}(0)
\]
as $j \to \infty$, where $D_{\infty}$ is a local model associated with $D$ at $p^0$.
\end{thm}

Now we consider the Levi corank one case. Recall that a boundary point $p^0$ of a domain $D \subset \mbb{C}^n$ is said to have Levi corank one if there exists a neighbourhood of $p^0$ where $\pa D$ is smooth, pseudoconvex, of finite type, and the Levi form has at least $(n-2)$ positive eigenvalues. If every boundary point of $D$ has Levi corank one, then $D$ is called a Levi corank one domain. This includes the class of {\it all} smoothly bounded pseudoconvex finite type domains in $\mathbb{C}^2$. A basic example in higher dimension is the egg
\[
E_{2m} = \Big\{ z \in \mbb C^n : \vert z_1 \vert^{2m} + \vert z_2 \vert^2 + \ldots + \vert z_n \vert^2 < 1\Big\}
\]
where $m\ge 2$ is an integer. In general, if $\rho$ is a local defining function for $D$ at a Levi-corank one boundary point $p^0$, then it was proved in \cite{Cho2} that for each point $p$ in a sufficiently small neighbourhood $U$ of $p^0$, there are holomorphic coordinates $\zeta=\Phi^p(z)$ such that
\begin{multline}\label{nrmlfrm}
\rho \circ (\Phi^{p})^{-1}(\z) = \rho(p) + 2 \Re \z_n + \sum_{ \substack{j + k \le 2m\\
                                                          j, k > 0}} a_{jk}(p) \z_1^j \ov \z_1^k + \sum_{\al=2}^{n-1} \vert \zeta_{\al}\vert^2 \\
+ \sum_{\al = 2}^{n - 1} \sum_{ \substack{j + k \le m\\
                                                          j, k > 0}} \Re \Big( \big(b_{jk}^{\al}(p) \z_1^j \ov \z_1^k \big) \z_{\al} \Big)+
                                                          O\big(\vert \z_n\vert \vert \z \vert+\vert \z_{*}\vert^2 \vert \z \vert+\vert \z_{*}\vert \vert \z_1\vert^{m+1}+\vert \z_1\vert^{2m+1}\big)
\end{multline}
where $\z_*=(0,\z_2,\ldots, \z_{n-1},0)$.
To construct the distinguished polydiscs around $p$, set
\begin{equation}\label{defn-A-B}
\begin{aligned}
A_l(p) &= \max\Big\{\big\vert a_{jk}(p)\big\vert : j + k = l\Big\}, \quad 2 \leq l \leq 2m,\\
B_{l'}(p) & = {\rm max} \Big\{ \big\vert b_{jk}^\alpha (p) \big\vert \; : \; j+k=l', \; 2 \leq \alpha \leq n-1\Big \}, \; 2 \leq l' \leq m.
\end{aligned}
\end{equation}
Now define for each $\delta > 0$, the special-radius
\begin{alignat}{3} \label{E46}
\tau(p, \delta) = \min \Big\{ \Big( \delta/ A_l(p) \Big)^{1/l}, \; \Big(\delta^{1/2}/B_{l'}(p) \Big)^{1/l'}  \; : \; 2 \le l \le 2m, \; 2 \leq l' \leq m   \Big \}.
\end{alignat}
It was shown in \cite{Cho2} that the coefficients $b_{jk}^\alpha$'s in the above definition of $\tau(p,\delta)$ are insignificant and may be ignored, so that 
\begin{equation}\label{tau-defn2}
\tau(p, \delta) = \min \Big\{ \Big( \delta/ A_l(p) \Big)^{1/l}  \;: \; 2 \le l \le 2m \Big\}.
\end{equation}
Set
\[
\tau_1(p, \delta) = \tau(p, \delta) = \tau, \tau_2(p, \delta) = \ldots = \tau_{n-1}(p, \delta) = \delta^{1/2}, \tau_n(p, \delta) = \delta.
\]
The distinguished polydiscs $Q(p, \delta)$ of Catlin are defined by
\[
Q(p,\delta)=\Big\{(\Phi^p)^{-1}(\z) : \vert \z_1 \vert < \tau_1(p,\de), \ldots, \vert \z_n\vert < \tau_n(p,\de)\Big\}.
\]
The scaling method (which is well known in this case and will be briefly explained later) shows that every sequence in $D$ that converges to $p^0 \in \pa D$ furnishes limiting domains
\begin{equation}\label{Lcoran1-mod}
D_{\infty} = \left\{ z \in \mbb C^n : 2 \Re z_n + P_{2m}(z_1, \ov z_1) + \sum_{\al=2}^{n-1} \vert z_{\al} \vert^2<0\right\}
\end{equation}
where $P_{2m}(z_1, \ov z_1)$ is a subharmonic polynomial of degree at most $2m$ ($m \ge 1$) without harmonic terms, $2m$ being the $1$-type of $\pa D$ at $p^0$. Such a limiting domain $D_{\infty}$ is called a local model associated with $D$ at $p^0$. By Proposition~4.5 of \cite{Yu2} and the remark at the bottom of page~605 of the same article, $D_{\infty}$ possesses a local holomorphic peak function at every boundary point. By Lemma~1 of \cite{BP}, there is a local holomoprhic peak function for $D_{\infty}$ at the point at infinity also. It follows that $D_{\infty}$ is complete hyperbolic (see \cite{G2}). Observe that the point $b=({}'0,-1)$ lies in every such $D_{\infty}$.

\begin{thm} \label{Lcr1}
Let $D=\{\rho<0\}$ be a smoothly bounded Levi corank one domain in $\mathbb{C}^n$ and $p^j \in D$ be a sequence converging to $p^0 \in \pa D$. Let $\de_j>0$ be such that $\ti p^j=(p^j_{1},\cdots, p^j_{n}+\de_j)$ is a point on $\pa D$. Then up to a subsequence,
\[
k_D(p^j) \prod_{\al=1}^n \tau_{\al}(\ti p^j,\de_j)^2 \to c(\rho,p^0)k_{D_\infty}(b)
\]
as $j \to \infty$, where $c(\rho,p^0)$ is a positive constant that depends only on $\rho$ and $p^0$, and $D_{\infty}$ is a local model associated with $D$ at $p^0$.
\end{thm}

We conclude the article by showing an efficacy of the quotient invariant in determining strong pseudoconvexity if its boundary behaviour is a priori known---a property enjoyed by the squeezing function and its dual the Fridman invariant as well. We refer the reader to the recent articles \cites{MV, NV} and the references therein for the definition and other relevant materials related to these two invariants. Let us denote the squeezing function for a domain $D$ by $s_D$ and the Fridman invariant by $h_D$. It was proved in \cite{Zim} that if $D$ is a bounded convex domain with $C^{2, \al}$ boundary for some $\al\in(0,1)$, then $D$ is strongly pseudoconvex if $s_D(z) \to 1$ as $z \to \pa D$. Mahajan and Verma \cite{MV} showed that if $D$ is a smoothly bounded convex domain or if $D$ is a smoothly bounded $h$-extendible domain (i.e., $D$ is a smoothly bounded pseudoconvex finite type domain for which the Catlin and D'Angelo multitypes coincide at every boundary point), then $D$ is strongly pseudoconvex if either $h_D(z) \to 0$ or $s_D(z) \to 1$ as $z \to \pa D$. We have the following analog for the quotient invariant:

\begin{thm}\label{q-appln}
For any positive integer $n$ and $\al \in (0,1)$, there exists some $\ep=\ep(n, \al)>0$ with the following property: If $D \subset \mf{C}^n$ is a bounded convex domain with $C^{2, \al}$ boundary and if
\[
q_D(p) \geq 1-\ep
\]
outside a compact subset of $D$, then $D$ is strongly pseudoconvex.
\end{thm}

\medskip

{\it Acknowledgements}: The authors would like to thank K. Verma for his support and encouragement. Some of the material presented here has been benefited from conversations that the first author had with G.P. Balakumar, S. Gorai, and P. Mahajan. We would like to thank them for their valuable comments and suggestions. We thank the anonymous referee for useful suggestions for improving the exposition herein, especially Theorem~1.3 and its proof are based on the ideas given by the referee.

\section{Regularity of the volume elements}
In this section we prove continuity of the volume elements that is required for the proofs of Theorems~\ref{cvx} and \ref{Lcr1}. The arguments are similar to the case of the Carath\'eodory-Reiffen and Kobayashi-Royden pseudometrics and we present them only for convenience. First, a few remarks. If $D \subset \mbb{C}^n$ is any domain and $p \in D$, then $c_D(p)$ is attained. Indeed, choose a sequence $\psi^j \in \mathcal{O}(D, \mbb{B}^n)$ such that $\psi^j(p)=0$ and $\vert \det (\psi^j)^{\prime}(p)\vert^2 \to c_D(p)$. By Montel's theorem, passing to a subsequence if necessary, $\psi^j$ converges uniformly on compact subsets of $D$ to a map $\psi \in \mathcal{O}(D, \ov{\mbb{B}^n})$. Since $\psi(p)=0$, by the maximum principle $\psi\in \mathcal{O}(D, \mbb{B}^n)$, and it follows that $c_D(p)=\vert \det \psi^{\prime}(p)\vert^2$. In particular, this implies that $c_D(p)$ is always finite. Note that $c_D(p)$ can vanish (for example if $D=\mbb{C}$), but is strictly positive if $D$ is not a Liouville domain. Likewise, if $D$ is taut then similar arguments as above shows that $k_D(p)$ is attained. Observe that $k_D(p)$ is finite for any domain $D$ because we can put a ball $B(p,r)$ inside $D$ and consequently $\phi(t)=rt+p$ is a competitor for $k_D(p)$, giving us $k_D(p) \leq r^{-2n}$. It is possible that $k_D(p)$ can also vanish but if $D$ is bounded, then by invoking Cauchy's estimates we see that $k_D(p)>0$. Similarly, if $D$ is taut, then also $k_D(p)>0$ as it is attained. We will call a  map $\psi \in \mathcal{O}(D, \mbb{B}^n)$ satisfying $\psi(p)=0$ and $\vert \det \psi^{\prime}(p)\vert^2=c_D(p)$ a Carath\'eodory extremal map for $D$ at $p$. Similarly, a Kobayshi extremal map for $D$ at $p$ is a map $\psi \in \mathcal{O}(\mbb{B}^n,D)$ with $\psi(0)=p$ and $\vert \det \psi^{\prime}(0)\vert^{-2}=k_D(p)$.

\begin{prop}\label{cont-vol}
Let $D \subset \mbb{C}^n$ be a domain. Then $c_D$ is continuous. If $D$ is taut, then $k_D$ is also continuous.
\end{prop}
\begin{proof}
We will show that $c_D$ is locally Lipschitz which of course implies that $c_D$ is continuous. Let $B(a,2r) \subset\subset D$ and fix $p, q \in B(a,r)$. Choose a Carath\'eodory extremal map $\psi$ for $D$ at $p$. Then
\begin{multline*}
c_D(p)-c_D(q) \leq \big\vert \det \psi^{\prime}(p)\big\vert^2 - \big\vert\det \psi^{\prime}(q) \big\vert^2 c_{\mbb{B}^n}\big(\psi(q)\big) \\
= \big\vert \det \psi^{\prime}(p)\big\vert^2- \frac{\big\vert\det \psi^{\prime}(q) \big\vert^2}{\big(1-\vert \psi(q) \vert^2\big)^{n+1}} \leq \big\vert \det \psi^{\prime}(p)\big\vert^2 - \big\vert\det \psi^{\prime}(q)\big \vert^2.
\end{multline*}
Since the distances of $p$ and $q$ to $\pa D$ is at least $r$, by Cauchy's estimates the right hand side is bounded above by $C_r \vert p-q\vert$ where $C_r$ is a constant that depends only on $r$. Thus we can interchange the role of $p$ and $q$ to have $\vert c_D(p)-c_D(q) \vert \leq C_r\vert p-q\vert$ that establishes local Lipschitz property of $c_D$.

For $k_D$, first we show that it is upper semicontinuous for any domain $D$. Let $p\in D$ and $\ep>0$. Then there exists $\phi\in\mathcal{O}(\mbb{B}^n,D)$ with $\phi(0)=p$ such that
\begin{equation}\label{usc-1}
\vert \det \phi^{\prime}(0) \vert^{-2} < k_D(p)+\ep.
\end{equation}
Let $0<r<1$ and set for $z\in D$,
\[
f^z(t)=\phi\big((1-r) t\big)+ (z-p), \quad t \in \mbb{B}^n.
\]
Since $\phi\big(B(0,1-r)\big)$ is a relatively compact subset of $D$, there exists $\de>0$ such that if $z \in B(p,\de)$, then $f^z \in \mathcal{O}(\mbb{B}^n, D)$. Also $f^z(0)=z$ and so $f^z$ is a competitor for $k_{D}(z)$. Therefore,
\[
k_D(z) \leq \big\vert \det (f^z)^{\prime}(0)\big\vert^{-2} =(1-r)^{-2n}\big\vert \det \phi^{\prime}(0)\big\vert^{-2}.
\]
Letting $r\to 0^+$ and using \eqref{usc-1}, we obtain that
\[
k_D(z)< k_D(p)+\ep
\]
for all $z \in B(p,\de)$ which proves the upper semicontinuity of $k_D$.

Next we assume that $D$ is taut and show that $k_D$ is lower semicontinuous. Let $p \in D$. If possible, assume that $k_D$ is not lower semicontinuous at $p$. Then $k_D(p)>0$ and there exist $\ep>0$, a sequence $p^j \to p$, such that
\[
k_D(p^j)< k_D(p)-\ep.
\]
Since $D$ is taut, there are Kobayashi extremal maps $g^j$ for $D$ at $p^j$. Again by tautness and the fact that $g^j(0)=p^j\to p \in D$, passing to a subsequence, $g^j$ converges uniformly on compact subsets of $\mbb{B}^n$ to a map $g \in \mathcal{O}(\mbb{B}^n,D)$. Therefore,
\[
k_D(p^j) = \big\vert \det (g^j)^{\prime}(0)\big\vert^{-2} \to \big\vert\det g^{\prime}(0)\big\vert^{-2}.
\]
But $g$ is a competitor for $k_D(p)$ and so $k_D(p) \leq \vert\det g^{\prime}(0)\vert^{-2}$. Thus we have
\[
k_D(p) \leq k_D(p)-\ep
\]
which is a contradiction. This proves the lower semicontinuity of $k_D$ and thus $k_D$ is continuous if $D$ is taut.
\end{proof}

\section{Proof of Theorem \ref{cvx}}
Let us recall the hypothesis of Theorem~\ref{cvx}: We are given a smoothly bounded convex finite type domain $D=\{\rho<0\}$ and a sequence $p^j \in D$ converging to $p^0 \in \pa D$.  Without loss of generality assume that $p^0=0$. The numbers $\ep_j$ are defined by $\ep_j=-\rho(p^j)$. The maps $U^{p^j,\ep_j}\circ T^{p^j,\ep_j}$ satisfy
\[
U^{p^j,\ep_j}\circ T^{p^j,\ep_j}(p^j)=0.
\]
\subsection{Scaling}
Consider the dilations
\[
\La^{p^j,\ep_j}(z)=\left(\frac{z_1}{\tau_1(p^j,\epsilon_j)},\ldots, \frac{z_n}{\tau_n(p^j,\epsilon_j)}\right).
\]
The scaling maps are the compositions $S^j=\La^{p^j,\ep_j} \circ U^{p^j,\ep_j}\circ T^{p^j,\ep_j}$ and the scaled domains are $D^j=S^j(D)$. Note that $D^j$ is convex and $S^j(p^j)=0 \in D^j$, for each $j$. It was shown in \cite{G} that the defining functions $\rho^j=\frac{1}{\ep_j} \rho \circ (S^j)^{-1}$ for $D^j$, after possibly passing to a subsequence, converge uniformly on compact subsets of $\mf{C}^n$ to
\[
\rho_{\infty}(z)=-1 + \Re\sum_{\al=1}^n b_{\al} z_{\al} + P_{2m}({}'z),
\]
where $b_{\al}$ are complex numbers and $P_{2m}$ is a real convex polynomial of degree less than or equal to $2m$. This implies that after passing to a subsequence if necessary, the domains $D^j$ converge in the local Hausdorff sense to $D_{\infty}=\{\rho_{\infty}<0\}$.

\subsection{Stability of the volume elements}

\begin{lem}\label{normality-cvx}
Let $\phi^j \in \mathcal{O}(\mbb{B}^n, D^j)$ and $\phi^j(0)=a^j \to a \in D_{\infty}$. Then $\phi^j$ admits a subsequence that converges uniformly on compact subsets of $\mbb{B}^n$ to a map $\phi \in \mathcal{O}(\mbb{B}^n, D_{\infty})$.
\end{lem}
\begin{proof}
By the arguments in the proof of Lemma~3.1 in \cite{G}, observe that the family $\phi^j$ is normal. Also, $\phi^j(0)=a^j \to a$. Hence, the sequence $\phi^j$ admits a subsequence, which we denote by $\phi^j$ itself, and which converges uniformly on compact subsets of $\mathbb{B}^n$ to a holomorphic map $\phi:\mathbb{B}^n\rightarrow \mathbb{C}^n$. We will now show that $\phi\in \mathcal{O}(\mbb{B}^n, D_{\infty})$.

Let $0<r<1$. Then $\phi^j$ converges uniformly on $B(0,r)$ to $\phi$, and so the sets $\phi^j(B(0,r)) \subset K$ for some fixed compact set $K$ and for all large $j$. Since $\rho^j(\phi^j(t))<0$ for $t\in B(0,r)$ and for all $j$, we have $\rho_{\infty}(\phi(t)) \leq 0$, or equivalently $\phi(B(0,r))\subset \overline{D}_{\infty}$. Since $r\in(0,1)$ is arbitrary, we have $\phi(\mathbb{B}^n)\subset \overline{D}_\infty$. Since $\phi(0)=a \in D_\infty$, and $D_{\infty}$ possesses a local holomorphic peak function at every boundary point (see \cite{G2}), the maximum principle implies that $\phi(\mathbb{B}^n)\subset D_\infty$.
\end{proof}

\begin{prop}\label{stability-cvx}
For any $a\in D_\infty$,
\[
\lim_{j\rightarrow \infty} k_{D^j}(a)= k_{D_\infty}(a),
\]
Moreover, this convergence is uniform on compact subsets of $D_\infty$.
\end{prop}
\begin{proof}
Assume that $k_{D^j}$ does not converge to $k_{D_\infty}$ uniformly on some compact subset $S\subset D_\infty$. Then there exist $\epsilon_0>0$, a subsequence of $k_{D^j}$ which we denote by $k_{D^j}$ itself, and a sequence $a^j\in S$ satisfying
\[
\big\vert k_{D^j}(a^j)-k_{D_\infty}(a^j)\big\vert>\epsilon_0
\]
for all large $j$. Since $S$ is compact, after passing to a subsequence if necessary, $a^j \to a \in S$. Since $D_\infty$ is complete hyperbolic, and hence taut, $k_{D_\infty}$ is continuous by Proposition~\ref{cont-vol}. Hence for all large $j$, we have
\[
\big\vert k_{D_\infty}(a^j)-k_{D_\infty}(a) \big\vert \leq \frac{\epsilon_0}{2}.
\]
Combining the above two inequalities we have
\begin{equation}\label{stability-cvx-1}
\big\vert k_{D^j}(a^j)-k_{D_\infty}(a)\big\vert>\frac{\epsilon_0}{2}
\end{equation}
for all large $j$. We will deduce a contradiction in the following two steps:

\textit{Step 1}. $\limsup_{j\rightarrow \infty}k_{D^j}(a^j)\leq k_{D_\infty}(a)$. Since $D_{\infty}$ is taut, we have $0<k_{D_\infty}(a)<\infty$ and there exists a Kobayashi extremal map $\psi$ for $D_{\infty}$ at $a$. Fix $0<r<1$ and define the holomorphic maps $\psi^j:\mathbb{B}^n\rightarrow \mathbb{C}^n$ by
\[
\psi^j(t)=\psi\big((1-r)t\big)+(a^j-a).
\]
Since the image $\psi\big(B(0,1-r)\big)$ is compactly contained in $D_\infty$ and $a^j\rightarrow a$ as $j\rightarrow \infty$, it follows that $\psi^j \in \mathcal{O}(\mathbb{B}^n, D^j)$ for all large $j$. Also, $\psi^j(0)=\psi(0)+a^j-a= a^j$ and thus $\psi^j$ is a competitor for $k_{D^j}(a^j)$. Therefore,
\[
k_{D^j}(a^j)\leq \big\vert \det (\psi^j)'(0)\big\vert^{-2} =(1-r)^{-2n} \big\vert \det \psi'(0)\big\vert^{-2}.
\]
Letting $r \to 0^+$, we get
\[
\limsup_{j\rightarrow \infty}k_{D^j}(a^j)\leq k_{D_\infty}(a).
\]

\textit{Step 2}. $k_{D_\infty}(a)\leq \liminf_{j\rightarrow \infty}k_{D^j}(a^j)$. Fix $\epsilon>0$ arbitrarily small. Then there exist $\phi^j \in \mathcal{O}(\mathbb{B}^n, D^j)$ such that $\phi^j(0)=a^j$ and
\begin{equation}\label{stability-cvx-2}
\big\vert \det (\phi^j)'(0)\big\vert^{-2} < k_{D^j}(a^j)+\epsilon.
\end{equation}
By Lemma~\ref{normality-cvx}, $\phi^j$ admits a subsequence which we denote by $\phi^j$ itself, and which converges uniformly on compact subsets of $\mbb{B}^n$ to a map $\phi\in \mathcal{O} (\mbb{B}^n, D_{\infty})$. Then from \eqref{stability-cvx-2}
\[
\big\vert \det \phi'(0)\big\vert^{-2} \leq \liminf_{j \to \infty} k_{D^j}(a^j)+\epsilon
\]
But $\phi$ is a competitor for $k_{D_\infty}(a)$ and $\ep$ is arbitrary. So we obtain
\[
k_{D_{\infty}}(a) \leq \liminf_{j \to \infty} k_{D^j}(a^j)
\]
as required.

By Step 1 and Step 2, we have $\lim_{j\rightarrow \infty}k_{D^j}(a^j)=k_{D_\infty}(a)$ which contradicts \eqref{stability-cvx-1} and thus the proposition is proved.
\end{proof}

We believe that the analog of the above stability result holds for the Carath\'eodory volume element also but we do not have a proof. However, we do have the following:

\begin{prop}\label{c-stability-cvx}
For $a^j \in D^j$ converging to $a\in D_\infty$,
\[
\limsup_{j\rightarrow \infty} c_{D^j}(a^j) \leq c_{D_\infty}(a),
\]
\end{prop}
\begin{proof}
If possible, assume that this is not true. Then there exists a subsequence of $c_{D^j}(a^j)$ which we denote by $c_{D^j}(a^j)$ itself, and an $\ep>0$, such that
\[
c_{D^j}(a^j) > c_{D_{\infty}}(a)+\ep, \quad \text{for all } j \geq 1.
\]
Let $\psi^j$ be a Carath\'eodory extremal map for $D^j$ at $a^j$. Since the target of these maps is $\mbb{B}^n$, passing to a subsequence if necessary, $\psi^j$ converges uniformly on compact subsets of $D_{\infty}$ to a holomorphic map $\psi: D_{\infty} \to \ov{\mbb{B}^n}$, and since $\psi(a)=0$ we must have $\psi\in \mathcal{O}(D_{\infty}, \mbb{B}^n)$. Now, the above inequality implies that this limit map satisfies
\[
\big\vert \det \psi'(a)\big\vert^2 \geq c_{D_{\infty}}(a)+\ep.
\]
On the other hand as $\psi$ is a candidate for $c_{D_{\infty}}(a)$, we also have
\[
c_{D_{\infty}}(a)\geq \big\vert \det \psi'(a)\big\vert^2.
\]
Combining the last two inequalities, we obtain
\[
c_{D_{\infty}}(a)\geq c_{D_{\infty}}(a)+\ep
\]
which is a contradiction.
\end{proof}

\subsection{Proof of Theorem~\ref{cvx}}
By the transformation rule
\[
k_D(p^j)=\big\vert \det (\Lambda^{p^j,\epsilon_j}U^{p^j,\epsilon_j}T^{p^j,\epsilon_j})'(p^j) \big\vert^2 k_{D^j}(0).
\]
Since $\vert \det (\La^{p^j,\epsilon_j})'(0)\vert^2=\prod_{\al=1}^n \tau_{\al}(p^j,\epsilon_j)^{-2}$ we get 
\[
k_D(p^j) \prod_{\al=1}^n \tau_{\al}(p^j,\epsilon_j)^{2}= k_{D^j}(0).
\]
Recall that the domains $D^j$ converge in the local Hausdorff sense to $D_{\infty}$ up to a subsequence and hence in view of Proposition~\ref{stability-cvx}, a limit of the right hand side is $k_{D_\infty}(0)$. This completes the proof of the theorem.

\section{Proof of Theorem \ref{Lcr1}}
\subsection{Change of coordinates}
Let $D=\{\rho<0\}$ be a smoothly bounded Levi corank one domain and $p^0 \in \pa D$. We may assume that the Levi form of $\rho$ at $p_0$ has exactly $n-2$ positive eigenvalues. We recall the definition of the change of coordinates $\Phi^p$ that transform $\rho$ into the normal form \eqref{nrmlfrm}. The maps $\Phi^p$ are actually holomorphic polynomial automorphisms defined as $\Phi^p=\phi_5 \circ \phi_4 \circ \phi_3 \circ \phi_2\circ \phi_1$ where $\phi_i$ are described below. Since the volume elements are invariant under unitary rotations, we assume without loss of generality that $\pa \rho/\pa z_n(p^0) \neq 0$. Then there is a neighbourhood $U$ of $p^0$ such that $(\pa \rho/\pa z_n)(p)\neq 0$ for all $p \in U$. Thus
\[
\nu=\left(\frac{\pa \rho}{\pa z_1}, \ldots, \frac{\pa \rho}{\pa z_n}\right)
\]
is a nonvanishing vector field on $U$. Note that the vector fields
\[
L_n=\frac{\pa}{\pa z_n}, \quad L_{\al}=\frac{\pa}{\pa z_{\al}}-b_{\al}\frac{\pa}{\pa z_n}, \quad 1 \leq \al \leq n-1,
\]
where $b_{\al}=\frac{\pa \rho}{\pa z_{\al}}/\frac{\pa \rho}{\pa z_n}$, form a basis of $T^{1,0}(U)$. Moreover, for $1 \leq \al \leq n-1$, $L_{\al}\rho\equiv 0$ and so $L_{\al}$ is a complex tangent vector field to $\pa D \cap U$. Shrinking $U$ if necessary, we also assume that
\[
\begin{bmatrix} \pa \ov \pa \rho (L_{\al}, \ov L_{\be})\end{bmatrix}_{2 \leq \al,\be \leq n-1}
\]
has all its eigenvalues positive at each $p \in U$.
\begin{enumerate}[(i)]
\item The map $\phi_1$ is defined by
\[
\phi_1(z)=\big(z_1-p_1, \ldots, z_{n-1}-p_{n-1}, \langle z-p, \nu(p)\rangle\big)
\]
and it normalises the linear part of the Taylor series expansion of $\rho$ at $p$. In the new coordinates which we denote by $z$ itself, $\rho$ takes the form
\[
\rho\circ \phi_1^{-1}(z)=\rho(p)+ 2 \Re z_n+ O(\vert z\vert^2).
\]
\item Now 
\[
A=\begin{bmatrix}\frac{\pa^2 \rho}{\pa z_{\al}\pa \ov z_{\be}}(p)\end{bmatrix}_{2 \leq \al,\be\leq n-1}
\]
is a Hermitian matrix and there is a unitary matrix $P=\begin{bmatrix}P_{jk}\end{bmatrix}_{2 \leq j,k\leq n-1}$ such that $P^*AP=D$, where $D$ is a diagonal matrix whose entries are the positive eigenvalues of $A$. Writing $\ti z=(z_2, \ldots z_{n-1})$, the map $w=\phi_2(z)$ is defined by
\begin{align*}
w_1=z_1, \quad w_n=z_n, \quad \ti w = P^T\ti z.
\end{align*}
Then
\[
\sum_{\al,\be=2}^{n-1} \frac{\pa^2 \rho}{\pa z_{\al}\pa \ov z_{\be}}(p) z_{\al} \ov z_{\be} = \ti z^T A \ov{\ti z}=(\ov P \ti w)^T A \ov{(\ov P \ti w)}= \ti w ^T D \ov{\ti w} = \sum_{\al=2}^{n-1} \la_{\al} \vert w_{\al} \vert^2,
\]
where $\la_{\al}>0$ is the $\al$-th entry of $D$. Thus, denoting the new coordinates $w$ by $z$ again,
\[
\rho\circ \phi_1^{-1}\circ\phi_2^{-1}(z)=\rho(p)+2\Re z_n + \sum_{\al=2}^{n-1} \la_{\al}\vert z_{\al}\vert^2+O(\vert z\vert^2)
\]
where $O(\vert z\vert^2)$ consists of only the non-Hermitian quadratic terms and all other higher order terms.
\item The map $w=\phi_3(z)$ is defined by $w_1=z_1, w_n=z_n$, and $w_j=\la_j^{1/2}z_j$ for $2 \leq j \leq n-1$. In the new coordinates, still denoted by $z$,
\begin{multline}\label{rho-exp3}
\rho \circ \phi_1^{-1}\circ\phi_2^{-1} \circ \phi_3^{-1}(z)= \rho(p)+2\Re z_n +\sum_{\al=2}^{n-1}\sum_{j=1}^m 2 \Re\big((a^{\al}_jz_1^j+b^{\al}_j\ov z_1^j)z_{\al}\big)\\
 + 2 \Re \sum_{\al=2}^{n-1} c_{\al}z_{\al}^2
 + \sum_{2\leq j+k \leq 2m}a_{jk}z_1^j\ov z_1^k
 +\sum_{\al=2}^{n-1} \vert z_{\al}\vert^2
 +\sum_{\al=2}^{n-1} \sum_{\substack{j+k\leq m\\j,k>0}} 2 \Re \big(b^{\al}_{jk}z_1^j \ov z_1^k z_{\al}\big)\\
 +O\big(\vert z_n\vert \vert z \vert+\vert z_{*}\vert^2 \vert z \vert+\vert z_{*}\vert \vert z_1\vert^{m+1}+\vert z_1\vert^{2m+1}\big)
\end{multline}
where $z_*=(0,z_2,\ldots, z_{n-1},0)$.

\item Next, the pure terms in \eqref{rho-exp3}, i.e., $z_{\al}^2$, $z_1^k$, $\ov z_1^k$, as well as $z_1^kz_{\al}$, $\ov z_1^k\ov z_{\al}$ terms are removed by absorbing them into the normal variable $z_n$ in terms of the change of coordinates $t=\phi_4(z)$ which is defined by
\begin{align*}
z_j & =t_j, \quad 1 \leq j \leq n-1,\\
z_n & =t_n -\hat{Q_1}(t_1, \ldots, t_{n-1}),
\end{align*}
where
\[
\hat{Q_1}(t_1, \ldots, t_{n-1}) = \sum_{k=2}^{2m}a_{k0}t_1^k-\sum_{\al=2}^{n-1}\sum_{k=1}^{m}a_k^{\al}t_{\al}t_1^k-\sum_{\al=2}^{n-1}c_{\al}t_{\al}^2.
\]

\item In the final step, the terms of the form $\ov t_1^j t_{\al}$ are removed by applying the transformation $\zeta=\phi_5(t)$ given by
\begin{align*}
t_1 & =\zeta_1, t_n=\zeta_n,\\
t_{\al} & = \z_{\al}-Q_2^{\al}(\z_1), \quad 2 \leq \al \leq n-1,
\end{align*}
where $Q_2^{\al}(\z_1)=\sum_{k=1}^m \ov b_k^{\al}\z_1^k$. In these coordinates, $\rho$ takes the normal form \eqref{nrmlfrm}.
\end{enumerate}
It is evident from the definition of $\Phi^p$ that $\Phi^p(p)=0$,
\[
\Phi^{p}(p_1, \ldots, p_{n-1}, p_n - \ep) = \left(0, \ldots, 0, -\ep \;\frac{\pa \rho}{\pa \ov z_n}(p)\right),
\]
and
\begin{equation}\label{der-phi-p}
\det (\Phi^p)^{\prime}(p)=\frac{\pa \rho}{\pa \ov z_n}(p) (\la_2 \cdots \la_{n-1})^{1/2},
\end{equation}
where $\la_2, \ldots, \la_{n-1}$ are the positive eigenvalues of
\[
\begin{bmatrix} \frac{\pa^2 \rho}{\pa z_{\al} \pa \ov z_{\be}}(p) \end{bmatrix}_{2 \leq \al,\be \leq n-1}.
\]
\subsection{Scaling} 
Suppose $p^0=0$ and $\rho$ is in the normal form (\ref{nrmlfrm}) for $p=p^0$; in particular, $\nu(p^0) = ('0,1)$. Let $p^j \in D$ be a sequence converging to $p^0$. The points $ \ti p^j \in \partial D $ are chosen so that $ \ti p^j = p^j + ('0, \de_j) $ for some 
$ \de_j > 0 $. Then $ \de_j \approx \delta_D(p^j) $, where $\de_D(p)=d(p,\pa D)$ is the distance of $p$ to the boundary of $D$. Here and henceforth by the notation $a \approx b$ for positive functions $a,b$ depending on several parameters, we mean that the ratio $a/b$ is bounded above and below by some uniform positive constants independent of the parameters. The polynomial automorphisms $ \Phi^{\ti p^j} $ of $ \mathbb{C}^n $ as described above satisfy $ \Phi^{\ti p^j}(\ti p^j) = ('0, 0) $ and 
\[
 \Phi^{\ti p^j} (p^j) = 
\big('0, - \de_j d_0(\ti p^j) \big) ,
\]
where $ d_0(\ti p^j ) = \partial \rho/\partial \overline{z}_n (\ti p^j) \rightarrow 1 $ as $ j \rightarrow \infty $. 
 
Define a 
dilation of coordinates by
\[
\Delta^{\ti p^j, \de_j} (z_1, z_2, \ldots, z_n ) = 
\left( \frac{z_1}{\tau(\ti p^j, \de_j)}, \frac{z_2}{\de_j^{1/2}}, \ldots, \frac{z_{n-1}}{\de_j^{1/2}}, \frac{z_n} {\de_j} \right).
\]
The scaling maps are $S^j= \Delta^{\ti p^j, \de_j} \circ \Phi^{\ti p^j}$ and the scaled domains are $D^j=S^j(D)$. Note that $D^j$ contains $S^j(p^j) = \left('0, - d_0(\ti p^j) \right)$ which we will denote by $b^j$ and which converges to $b=('0,-1)$. From \eqref{nrmlfrm}, the defining function $\rho^j=\frac{1}{\de_j} \rho \circ (S^j)^{-1}$ for $D^j$ has the form
\begin{align*}
\rho^j(z)=2\Re z_n+P^j(z_1, \ov z_1)+\sum_{\al=2}^n \vert z_{\al}\vert^2+\sum_{\al=2}^{n-1} \Re\big(Q^{j}_{\al}(z_1, \ov z_1)z_{\al}\big)+O(\tau_1^{j}),
\end{align*}
where $\tau_1^j=\tau_1(\ti p^j, \de_j)$,
\[
P^j(z_1, \ov z_1)=\sum_{\substack{\mu + \nu \le 2m\\
                                                          \mu, \nu> 0}}a_{\mu\nu}(\ti p^j)\de_j^{-1}(\tau_1^j)^{\mu+\nu}z_1^{\mu}\ov z_1^{\nu},
\]
and
\[
Q^j_{\al}(z_1, \ov z_1)=\sum_{\substack{\mu + \nu \le m\\
                                                          \mu, \nu > 0}}b^{\al}_{\mu\nu}(\ti p^j)\de_j^{-1/2}(\tau_1^j)^{\mu+\nu}z_1^{\mu}\ov z_1^{\nu}.
\]
By \eqref{defn-A-B} and the definition of $\tau_1$, the coefficients of $P^j$ and $Q^j_{\al}$ are bounded by $1$. By Lemma~3.7 in \cite{TT}, it follows that the defining functions $\rho^j$, after possibly passing to a subsequence, converge together with all derivatives uniformly on compact subsets of $\mf{C}^n$ to
\[
\rho_{\infty}(z)= 2 \Re z_n + P_{2m}(z_1, \ov z_1) + \sum_{\al=2}^{n-1} \vert z_{\al} \vert^2,
\]
where $P_{2m}(z_1, \ov z_1)$ is a polynomial of degree at most $2m$ without harmonic terms. This implies that the corresponding domains $D^j$ converge in the local Hausdorff sense to $D_{\infty} = \{\rho_{\infty}<0\}$. Note that since $D_{\infty}$ is a smooth limit of pseudoconvex domains, it is pseudoconvex and hence $P_{2m}$ is subharmonic.

\subsection{Stability of the volume elements}
\begin{lem}\label{normality2}
Let $\phi^j \in \mathcal{O}(\mbb{B}^n, D^j)$ and $\phi^j(0)=a^j \to a \in D_{\infty}$. Then $\phi^j$ admits a subsequence that converges uniformly on compact subsets of $\mbb{B}^n$ to a map $\phi\in \mathcal{O}(\mbb{B}^n, D_{\infty})$.
\end{lem}

\begin{proof}
We first claim that the sequence $q^j:=(S^j)^{-1}(a^j) \in D$ converges to $p^0 \in \pa D$, where $p^0=0$ is the base point for scaling. Choose a relatively compact neighbourhood $K$ of $a$ in $D_{\infty}$. Since $a^j \to a \in D_{\infty}$, $a^j \in K$ for all large $j$. Now choose a constant $C>1$ large enough, so that $K$ is compactly contained in the polydisc
\[
\De(0, C^{1/2m}) \times \De(0,C^{1/2}) \cdots \De(0, C^{1/2})\times \De(0, C).
\]
From \eqref{tau-defn2}, we have $\tau_1(\ti p^j, C\de_j) \geq C^{1/2m} \tau_1(\ti p^j, \de_j)$. Moreover, by definition,
\[
\tau_{\al}(\ti p^j, C\de_j)=(C\de^j)^{1/2}=C^{1/2} \tau_{\al}(\ti p^j, \de_j)
\]
for $\al=2,\ldots, n-1$, and
\[
\tau_n(\ti p^j, C\de_j)=C\de_j=C\tau_n(\ti p^j, \de_j).
\]
As a consequence, the above polydisc is contained in
\[
\prod_{\al=1}^n \De\left(0, \frac{\tau_{\al}(\ti p^j, C\de_j)}{\tau_{\al}(\ti p^j, \de_j)}\right).
\]
The pull back of this polydisc by $S^j=\De^{\ti p^j,\de_j}\circ \Phi^{\ti p^j}$ is precisely $Q(\ti p^j, C\de_j)$. Thus
\[
q^j \in Q(\ti p^j, C\de_j)
\]
for all large $j$. Since $\ti p^j \to p^0$ and $\de_j \to 0$ as $j \to \infty$, it follows that $q^j \to p^0$ establishing our claim.

Now we prove that the family $\phi^j$ is normal. Consider the sequence of maps
\[
f^j=(S^j)^{-1}\circ \phi^j : \mbb{B}^n \to D.
\]
Note that $f^j(0)=q^j \to p^0$. By the arguments similar to the proof of Theorem~3.11 in \cite{TT} (also see Proposition~1 of \cite{Be}), for every $0<r<1$, there exists a constant $C_r$ depending only on $r$ such that
\[
f^j\big(B(0,r)\big) \subset Q(\ti p^j, C_r\de_j)
\]
for all large $j$. This implies that
\[
\phi^j\big(B(0,r)\big) \subset \prod_{\al=1}^n \De\left(0, \frac{\tau_{\al}(\ti p^j, C_r\de_j)}{\tau_{\al}(\ti p^j, \de_j)}\right)
\]
for all large $j$. Again from \eqref{tau-defn2}, $\tau_1(\ti p^j, C_r\de_j) \leq C_r^{1/2} \tau_1(\ti p^j, \de_j)$. Together with this, using the definition of $\tau_{\al}$ for $\al=2, \cdots n$, we see that the above polydisc is contained in
\[
\De\big(0, C_r^{1/2}\big) \times \cdots \times  \De\big(0, C_r^{1/2}\big) \times \De\big(0, C_r\big).
\]
Using a diagonal argument, it now follows that the family $\phi^j$ is normal.

Now, since $\phi^j(0)=a^j \to a \in D_{\infty}$, $\phi^j$ admits a subsequence which we denote by $\phi^j$ itself and which converges uniformly on compact subsets of $\mbb{B}^n$ to a holomorphic mapping $\phi:\mathbb{B}^n\rightarrow \mathbb{C}^n$. Since $D_{\infty}$ possess a local holomorphic peak function at every boundary point \cite[Proposition~4.5 and the remark in page 605]{Yu2}, arguments similar to Lemma~\ref{normality-cvx} now implies that $\phi(\mathbb{B}^n)\subset D_\infty$.
\end{proof}

With this lemma, the proof of the following proposition is exactly similar to that of Proposition~\ref{stability-cvx} and so we do not repeat the arguments.
\begin{prop}\label{stability-levicorank1}
For any $a\in D_\infty$,
\[
\lim_{j\rightarrow \infty} k_{D^j}(a)= k_{D_\infty}(a),
\]
Moreover, this convergence is uniform on compact subsets of $D_\infty$.
\end{prop}

Similarly, the proof of Proposition~\ref{c-stability-cvx} also gives
\begin{prop}\label{c-stability-levicorank1}
For $a^j \in D^j$ converging to $a\in D_\infty$,
\[
\limsup_{j\rightarrow \infty} c_{D^j}(a^j) \leq c_{D_\infty}(a).
\]
\end{prop}

\subsection{Proof of Theorem~\ref{Lcr1}}
Recall that we are in the case when $p^0=0$ and $\rho$ is in the normal form for $p=p^0$. Therefore, $\Phi^{p^0}=I$, the identity map. Observe that by the transformation rule
\[
k_D(p^j)=\big\vert \det\big(S^j\big)'(p^j)\big\vert^2 k_{D^j}(b^j),
\]
where $S^j= \Delta^{\ti p^j, \de_j} \circ \Phi^{\ti p^j}$ are the scaling maps. Since
\[
\Big\vert \det (\De^{\ti p^j,\de_j})' \big(\Phi^{\ti p^j}(p^j)\big)\Big\vert^2=\prod_{\al=1}^n \tau_{\al}(\ti p^j,\de_j)^{-2},
\]
we get
\begin{equation}\label{v-D-Dj}
k_{D}(p^j)\prod_{\al=1}^n \tau_{\al}(\ti p^j,\de_j)^{2}=\big\vert \det (\Phi^{\ti p^j})'(p^j)\big\vert^2 k_{D^j}(b^j).
\end{equation}
Now $\vert \det (\Phi^{\ti p^j})'(p^j)\vert \to \vert \det \big(\Phi^{p^0}\big)^{\prime}(p^0)\vert = 1$, and recall that after possibly passing to a subsequence, the domains $D^j$ converge in the local Hausdorff sense to $D_{\infty}$. Hence by Propostion~\ref{stability-levicorank1}, the right hand side of \eqref{v-D-Dj} has $k_{D_{\infty}}(b)$ as a limit, proving the theorem in the current situation.

For the general case, assume that $(\pa \rho/\pa z_n)(p^0)\neq 0$ and make an initial change of coordinates $w=T(z)=\Phi^{p^0}(z)$. Let $\Om =T(D)$, $q^0=T(p^0)=0$, and $q^j=T(p^j)$. Then
\begin{equation}\label{vD-vOm}
k_{D}(p^j)=\big\vert \det T'(p^j)\big\vert^2 k_{\Om}(q^j).
\end{equation}
To emphasise the dependence of $\Phi^p$, $\tau$, and $\tau_{\al}$ on $D=\{\rho<0\}$, we will write them now as $\Phi^p_\rho$, $\tau_\rho$ and $\tau_{\al,\rho}$ respectively.
Note that the defining function $r=\rho\circ T^{-1}$ for $\Om$ is in the normal form at $q^0=0$. Choose $\eta_j$ such that $\ti q^j=(q^j_1, \ldots, q^j_{n-1},q^j_n+\eta_j) \in \pa \Om$. Then by the previous case
\begin{equation}\label{v-Om}
k_{\Om} (q^j)\prod_{\al=1}^n \tau_{\al,r}(\ti q^j, \eta_j)^2 \to k_{D_{\infty}}(b)
\end{equation}
up to a subsequence. Since $\de_{\Om} \circ T$ is a defining function for $D$, we have $\de_{\Om}\circ T \approx \de_D$ and hence $\de_j \approx \de_D(p^j) \approx \de_{\Om}(q^j) \approx \eta_j$. Also, by (3.3) of \cite{TT},
\[
\rho \circ (\Phi_\rho^{p^j})^{-1} = r\circ (\Phi_r^{q^j})^{-1}.
\]
It follows from (2.9) of \cite{Cho2} that $\tau_{\rho}(\ti p^j, \de_j) \approx \tau_r(\ti q^j, \eta_j)$. Hence, after passing to a subsequence if necessary,
\begin{equation}\label{tauD-tauOm}
\prod_{\al=1}^n \frac{\tau_{\al,\rho}(\ti p^j,\de_j)}{\tau_{\al,r}(\ti q^j,\eta_j)} \to c_0
\end{equation}
for some $c_0>0$ that depends only on $\rho$. Also,
\begin{equation}\label{limit-der-T}
\big\vert \det T'(p^j)\big\vert \to \big\vert \det T'(p^0)\big\vert = \left \vert \frac{\pa \rho}{\pa \ov z_n}(p^0)\right\vert\prod_{\al=2}^{n-1} \la_{\al}^{1/2},
\end{equation}
by \eqref{der-phi-p}, where $\la_{\al}$'s are the positive eigenvalues of
\[
\begin{bmatrix}
\frac{\pa^2 \rho}{\pa z_{\al}\pa \ov z_{\be}}(p^0)
\end{bmatrix}_{2 \leq \al,\be\leq n-1}.
\]
It follows from \cref{vD-vOm,v-Om,tauD-tauOm,limit-der-T} that
\[
k_{D}(p^j) \prod_{i=1}^n \tau_{\al,\rho}(\ti p^j,\de_j)^2 \to c_0^2\left \vert \frac{ \pa \rho}{\pa \ov z_n}(p^0)\right\vert^2 \left(\prod_{\al=2}^{n-1} \la_{\al} \right)\, k_{D_{\infty}}(b)
\]
up to a subsequence. This completes the proof of the theorem.

\section{Proof of Theorem~1.3}
A convex domain $D \subset \mbb{C}^n$ is called $\mbb{C}$-properly convex if it does not contain any affine complex lines. Let $\mbb{X}_n$ denote the set of all $\mbb{C}$-properly convex domains endowed with the local Hausdorff topology. Consider the space
\[
\mbb{X}_{n,0}=\big\{(D, p) : D \in \mbb{X}_n, p \in D\big\} \subset \mbb{X}_n \times \mbb{C}^n
\]
endowed with the subspace topology. It was shown in \cite{Barth} that a convex domain in $\mbb{C}^n$ is complete hyperbolic if and only if it is $\mbb{C}$-properly convex. In particular, $\mbb{C}$-properly convex domains are taut and hence the quotient invariant on such domains are well-defined. Thus we have a function $q: \mbb{X}_{n,0} \to \mf{R}$ defined by
\[
q(D, p)=q_D(p).
\]
Recall that a function $f: \mbb{X}_{n,0} \to \mf{R}$ is called intrinsic (see \cite{Zim16}) if $f(D,p)=f(D',p')$ whenever there exits a biholomorphism $F: D \to D'$ with $F(p)=p'$. Thus the function $q$ is intrinsic. The following theorem was proved by Zimmer:
\begin{thm}[\cite{Zim19}]\label{Zim-norm}
Let $f: \mbb{X}_{n,0} \to \mf{R}$ be an upper semicontinuous intrinsic function with the following property: if $D\in \mbb{X}_{n}$ and $f(D,p) \geq f(\mf{B}^n, 0)$ for all $p \in D$, then $D$ is biholomorphic to $\mf{B}^n$. Then for any $\al>0$, there exists some $\ep=\ep(n,f, \al)>0$ such that: if $D \subset \mf{C}^n$ is a bounded convex domain with $C^{2, \al}$ boundary and
\[
f(D,p) \geq f(\mf{B}^n, 0) -\ep
\]
outside some compact subset of $D$, then $D$ is strongly pseudoconvex.
\end{thm}
Observe that if $D \subset \mf{C}^n$ is any domain and if $q_D(p) \geq 1$ for some point $p \in D$, then $q_D(p)=1$ and so $D$ must be biholomorphic to $\mf{B}^n$. Thus, to prove Theorem~1.3, we only need to show that the function $q: \mbb{X}_{n,0} \to \mf{R}$ is upper semicontinuous.

\begin{lem}\label{normality-hc}
Suppose $(D^j,p^j) \to (D_{\infty},p)$. If $f^j : \mf{B}^n \to D^j$, $f^j(0)=p^j$, then passing to a subsequence, $f^j$ converges uniformly on compact subsets of $D_{\infty}$ to a holomorphic function $f$ on $D_{\infty}$.
\end{lem}

This is precisely Lemma~3.1 of \cite{Zim16} with $\De$ replaced by $\mf{B}^n$ and since the proof is exactly same we do not repeat it.

\begin{prop}\label{usc}
The function $q: \mbb{X}_{n,0} \to \mf{R}$ is upper semicontinuous.
\end{prop}

\begin{proof}
The proof is very similar to the proof of Proposition~\ref{stability-cvx} and so we only outline it. Let $(D^j, a^j) \to (D_{\infty},a)$. The Step 1 of Proposition~\ref{stability-cvx} holds without any change. In view of Lemma~\ref{normality-hc}, Step 2 also holds. This implies that
\[
\lim_{j \to \infty}k_{D^j}(a^j)=k_{D_{\infty}}(a).
\]
The proof of Proposition~\ref{c-stability-cvx} goes through without any change and thus we have
\[
\limsup_{j \to \infty} c_{D^j}(a^j) \leq c_{D_{\infty}}(a).
\]
It follows that
\[
\limsup_{j \to \infty} q_{D^j}(a^j) \leq q_{D_\infty}(a)
\]
establishing the upper semicontinuity of $q$.
\end{proof}
Thus we have shown that $q$ satisfies the hypothesis of Theorem~\ref{Zim-norm} and this  completes the proof of Theorem~1.3.

\end{document}